\documentclass[a4paper,12pt]{article}
\usepackage[T2A]{fontenc}
\usepackage[utf8]{inputenc}

\usepackage[american]{babel}

\usepackage{anysize}

\usepackage{amsmath}
\usepackage{amssymb}
\usepackage{latexsym}
\usepackage{mathrsfs}

\usepackage{amsthm}

\usepackage{indentfirst}

\usepackage{color}

\usepackage{ifpdf}

\ifpdf
\usepackage[pdftex]{graphicx}
\else
\usepackage{graphicx}
\fi
\usepackage{accents}

\newtheorem{ttheorem}{Theorem}

\newtheorem{theorem}{Theorem}[section]
\newtheorem{lemma}[theorem]{Lemma}
\newtheorem{corollary}[theorem]{Corollaty}

\newtheorem{claim}[theorem]{Claim}

\usepackage[hidelinks]{hyperref}
\theoremstyle{definition}
\newtheorem{definition}{Definition}[section]

\theoremstyle{remark}
\newtheorem{remark}{Remark}
\newtheorem{example}[remark]{Example}

\newcommand{\bs}{\mathop{\rm BS}\nolimits}
\newcommand{\conv}{\mathop{\rm conv}\nolimits}

\begin{document}
\sloppy
\ifpdf
\DeclareGraphicsExtensions{.pdf, .jpg, .tif, .mps}
\else
\DeclareGraphicsExtensions{.eps, .jpg, .mps}
\fi

\title{Balanced Sets and Homotopy Invariants of Covers}
\author{Mikhail V. Bludov$^{1,2}$ }
\date{%
    $^1$HSE University, Russian Federation\\%
    $^2$MIPT, Russian Federation \\%
}

\maketitle


\begin{abstract}
In this paper, we study a construction of homotopy invariants of open or closed covers, where the homotopy class is defined relative to a pair $(V,r)$, with $V$ a finite set of points in $\mathbb{R}^d$ and $r$ a point in the interior of their convex hull. We show that the simplicial complex of non-balanced subsets associated with $(V,r)$ has the homotopy type of a sphere, and use this to develop a theory of homotopy invariants of covers relative to balanced sets. A key result is that the homotopy class of a cover depends only, up to an involution, on the balanced-equivalence class of $(V,r)$. As applications, we obtain extension theorems for covers in this setting and derive the KKMS lemma, its analogues, and related combinatorial fixed-point results.
 \footnote{The paper was prepared within the framework of the HSE University Basic Research Program}
\end{abstract}

\section{Introduction}

Suppose $T$ is a Hausdorff paracompact topological space, and let $\mathcal{F} = \{F_1, \dots, F_{n+1}\}$ be an open cover of $T$. It is well known (see, for example, Section~4G of \cite{Hatcher}) that, using a partition of unity, one can associate to $\mathcal{F}$ a map $\rho_{\mathcal{F}} : T \to \Delta^{n}$,
where $\Delta^n = \conv\{v_1, \dots, v_{n+1}\}$ is the $n$-dimensional simplex. If the intersection of all sets in $\mathcal{F}$ is empty, then the image of $\rho_{\mathcal{F}}$ lies in the boundary $S^{n-1}$ of the simplex. Thus the cover $\mathcal{F}$ determines a homotopy class of maps $[\rho_{\mathcal{F}}] \in [T, S^{n-1}]$. In \cite{MusH} it was shown that this class is independent of the choice of partition of unity, and that for closed covers a homotopy class can be defined in a similar fashion.

When $T$ is the $k$-dimensional sphere $S^k$, the homotopy class $[\rho_{\mathcal{F}}]$, when nontrivial, serves as an obstruction to extending the cover $\mathcal{F}$ to the ball $B^{k+1}$. Specifically, if $[\rho_{\mathcal{F}}] \neq 0$ and $G$ is an extension of $\mathcal{F}$ to $B^{k+1}$, then the intersection of all sets in $G$ is nonempty.

Moreover, if $T$ is an oriented closed manifold of dimension $n-1$ such that $T$ is the boundary of a manifold $N$, then this homotopy class is the mapping degree and acts as an obstruction to extending the cover to $N$. Overall, covering extension theorems in \cite{MusH} yield significant generalizations of the KKM lemma and Sperner's lemma.




In \cite{MusH}, a generalized definition of the homotopy class of a cover $\mathcal{F}$ was proposed, where instead of the map $\rho_{\Phi}: T \rightarrow \Delta^{n}$, the map $\rho_{\Phi,V}: T \rightarrow \conv(V)$ is considered, where $V=\{v_1,\dots,v_{n+1}\}$ is a set of points in $\mathbb{R}^d$, and the homotopy class of the cover is defined as a homotopy class of the map $\rho_{\mathcal{F},V}$ relative to a certain point $p \in \mathbb{R}^d$.

\medskip

Balanced subsets is an interesting combinatorial object that first appeared in cooperative game theory through the work of Bondareva \cite{Bon} and Shapley \cite{Sh67}, specifically in the context of the nonemptiness condition for the core. Here we will give the geometric definition of this object.

\begin{definition}
Suppose we have a set of points $V=\{v_1,\dots,v_{m+1}\}$ in $\mathbb{R}^d$ and an arbitrary point $r \in \mathbb{R}^d$. Then a subset $S\subset V$ is called $r$-balanced if $r \in \conv(S)$.
\end{definition}
We are primarily interested in the combinatorics of $r$-balanced subsets, so it is natural to speak of \emph{$\bs$-equivalent pairs}. Specifically, a pair $(V,r)$ in $\mathbb{R}^n$ is said to be $\bs$-equivalent to a pair $(V',r')$ in $\mathbb{R}^d$ if 
\[
S = \{v_{i_1}, \dots, v_{i_k}\} \in \bs(V,r) 
\quad \Longleftrightarrow \quad 
S' = \{v'_{i_1}, \dots, v'_{i_k}\} \in \bs(V',r').
\]

For convenience of notation and exposition, unless said otherwise, we fix a set of points  $V = \{v_1, \dots, v_m\} \subset \mathbb{R}^d$ 
with $\dim(\conv(V)) = d$. Additionally, we always assume that the covering $\mathcal{F} = \{F_1, \dots, F_m\}$ consists of $m$ sets, and each $F_i$ is associated with the corresponding point $v_i$.

Since non-balancedness is a decreasing property, for any finite set $V$ and any point $r \in \mathbb{R}^d$ we may define the \emph{simplicial complex of non-balanced subsets}. Specifically, let $\mathcal{K}(V,r)$ denote the simplicial complex with vertex set $\{v_1, \dots, v_m\}$ and
\[
\mathcal{K}(V,r) = \{ S \subset V \mid r \notin \conv(S) \}.
\]
We call simplices not in $\mathcal{K}(V,r)$ \emph{balanced}. The following theorem describes the homotopy type of the non-balanced complex:

\begin{ttheorem}
\label{NonBalComp1}
   Let $r$ be a point from the interior of $\conv(V)$. Then $|\mathcal{K}(V,r)|$ is homotopy equivalent to a sphere $S^{d-1}$
   
\end{ttheorem}

This theorem also appears in \cite{Blag} in a somewhat relevant context (there, non-balanced complexes are called zero-avoiding) and is used to establish the colorful Carathéodory theorem and its generalizations. The Alexander dual complexes of non-balanced (zero-avoiding) complexes are also examined in \cite{Blag}.

 In \cite{Sh}, Shapley proposed and proved the KKMS theorem, which connects a closed cover of a simplex to balanced sets. This connection was later extended by Komiya, who showed that covers of polytopes are closely related to balanced subsets of their vertices. Balanced 2-element subsets and the relationship between classical combinatorial fixed-point theorems and balanced subsets of polytopes were studied in \cite{BluMus}, while unbalanced sets were analyzed in \cite{Bil}. While there are several equivalent definitions of balanced subsets, throughout this paper we work only with the geometric definition of balanced subsets.

\medskip

In Section~4 we explain how to associate a homotopy class $[\mathcal{F}_{(V,r)}]$ with a covering $\mathcal{F}$ relative to a pair $(V,r)$. Typically, $[\mathcal{F}_{(V,r)}]$ is either an element of a homotopy group of a sphere or a mapping degree. We then explore connections between the combinatorial properties of balanced sets of points and the homotopy invariants of covers. In particular, we generalize the cover extension theorems from \cite{MusH} to the balanced case:

\begin{ttheorem}
  \label{VExtSph}

  Let $\mathcal{F}=\{F_1, \dots, F_m\}$ be a covering of a sphere $\mathbb{S}^{k}=\partial B^{k+1}$ such that there are no $r$-balanced simplices in the nerve $N(\mathcal{F})$
  Then the covering $\mathcal{F}$ can be extended to a covering $\mathcal{G}$ of the ball $\mathbb{B}^{k+1}$ such that there are no $r$-balanced simplices in the nerve $N(\mathcal{G})$ if and only if $[\mathcal{F}_{(V,r)}]=0$ in  $\pi_{k}(\mathbb{S}^{d-1})$.
\end{ttheorem}

and for the case of $n$-manifold with boundary:

\begin{ttheorem}
    \label{VExtMan}
    Let $M$ be an oriented manifold of dimension $d$ with boundary $N=\partial M$. Let $\mathcal{F}=\{F_1, \dots, F_m\}$ be a cover of $N$ such that there are no $r$-balanced simplices in $N(\mathcal{F})$ Then $\mathcal{F}$ can be extended to a cover $\mathcal{G}=\{G_1. \dots, G_n\}$ of the manifold $M$ such that there are no $r$-balanced simplices in $N(\mathcal{G})$ if and only if $\text{deg}(\mathcal{F}_{(V,r)})=0$.
\end{ttheorem}

It turns out that in this general case, the proofs require significant additions and rely on Theorem \ref{NonBalComp1}.

\medskip

The 'balanced' definition of the homotopy class of a cover depends on the choice of the set of points $V=\{v_1, \dots, v_m\}$ and the point $r$ in $\mathbb{R}^d$. In Theorem \ref{IndTh}, we show that the homotopy class of the cover does not depend, modulo an automorphism of order 2 on the set $[T, S^{d-1}]$, on the choice of the pair $(V,r)$ from the class of balanced-equivalent pairs, which we consider as the main result of this work.

We also study covers of Euclidean space $\mathbb{R}^d$ by a finite number of sets and define the index of a balanced intersection, which resembles the index of a singular point in a vector field.

\medskip

This paper is organized as follows. In Section 2 we will present the formal definitions and results from the work \cite{MusH}. In Section 3 we explore properties of balanced subsets; in particular, Theorem \ref{NonBalComp1} is proven. In Section 4 we discuss the 'balanced' definition of a homotopy class of a cover and prove the main theorems of this work. In Section 5 we discuss covers of Euclidean space $\mathbb{R}^n$. In Section 6 we discuss the relation of this theory with the classical discrete fixed-point theorems like Sperner's lemma and the KKM lemma. 

\bigskip

\section{Homotopy classes of covers}

Let's start by considering colorings of simplicial complexes. Let $V(K)$ be the set of vertices of a simplicial complex $K$. A vertex coloring by $m$ colors is a mapping $L: Vert(K) \rightarrow [m]$.

Let $\Delta^{m-1}$ be an $m-1$ simplex with vertices $V(\Delta^{m-1}) = \{v_1, \dots, v_m\}$. Consider $$f_L: Vert(K) \rightarrow Vert(\Delta^{m-1}),$$ $$f_L(u)=v_{L(u)}$$ for $u \in Vert(K)$. It induces a simplicial mapping $f_L: |K| \rightarrow| \Delta^{m-1}|$

Note that if there are no simplices in $K$ colored in all $m$ colors, then $f_L: |K| \rightarrow S^{m-2}$, where $S^{m-2}=\partial|\Delta^{m-1}|$, is well defined.

\begin{definition}[Definition 1.1. from \cite{MusH}]
    Let $K$ be a simplicial complex and let $L: Vert(K) \rightarrow [m]$ be a coloring of its vertices such that $K$ has no simplices colored in all $m$ colors. Then we denote by $[L]$ the homotopy class $[f_L] \in [K, \mathbb{S}^{m-2}]$.
\end{definition}

\begin{example}
Let $K$ be a triangulation of the sphere $\mathbb{S}^{k}$, and let $L: Vert(K) \rightarrow [m]$ be such that $K$ has no simplices colored in all $m$ colors. Then $[L] \in \pi_k(\mathbb{S}^{m-1})$.

If $k = m - 2$, then $[L] = deg(f_L) \in \mathbb{Z}$, where $deg(f)$ is the degree of the mapping from the sphere to itself.
\end{example}

\begin{example}
\label{Hopf}
    Let $\mathbb{S}^{3}_{12}$ be a 3-dimensional sphere with a triangulation consisting of 12 vertices, and let $\mathbb{S}^{2}_{4}$ be a 2-dimensional sphere with a triangulation consisting of 4 vertices $\{v_1, v_2, v_3, v_4\}$ (a tetrahedron). 
    In the work \cite{MadSar}, a simplicial map $\tau_{r}: \mathbb{S}^{3}_{12} \rightarrow \mathbb{S}^{2}_{4}$ was considered, such that $[\tau_{r}]=1$ in $ \pi_{3}(\mathbb{S}^{2})=\mathbb{Z}$. Then, for $u \in Vert(\mathbb{S}^{3}_{12})$, set $L_{\tau_{r}}(u) = i$ if $\tau_{r}(u) = v_i$. Thus, $[L_{\tau_{r}}] = 1 \in \pi_{3}(\mathbb{S}^{2})$.
\end{example}

In general, a coloring of a triangulation is a special case of a cover. For each vertex $u \in Vert(K)$, define its open star $St(u)$ as $|S| \setminus |B|$, where $S \subset K$ is the set of all simplices in $K$ that contain the vertex $u$, and $B$ is the set of all simplices in $K$ that do not contain the vertex $u$. Suppose we have a coloring $L: Vert(K) \rightarrow [m]$. Let ${W_l=\{u \in Vert(K) \mid  L(u)=l\}}$ be the set of the vertices of color $l$. Define an open set $U_l(K)$ as
$$
U_{l}(K) = \bigcup\limits_{u \in W_l} St(u).
$$ 
Thus, we obtain a cover $\mathcal{U}_L(K) = \{U_1(K), \dots, U_m(K)\}$ of $|K|$ by open sets.

\medskip

Now let's take a closer look at covers. In this work, we will only consider finite covers.

Let $T$ be a topological space, and suppose $\mathcal{F} = \{F_1, \dots, F_m\}$ is a family of open subsets of $T$ such that $\bigcup_{i=1}^mF_i=T$ i.e., $\mathcal{F}$ is an open cover. Let $N(\mathcal{F})$ be the nerve of the cover, i.e. a simplicial complex with vertices $\{1, \dots, m\}$ and simplices corresponding to intersections of sets, specifically, $\sigma \in N(\mathcal{F}) \iff \bigcap\limits_{i \in S} F_i \neq \emptyset$. We denote the geometric realization of the simplicial complex by $|N(\mathcal{F})|$. We will use $v_i$ to denote the vertex of the geometric realization corresponding to the vertex $i$.

A partition of unity $\Phi$ subordinate to the open cover $\mathcal{F}$ is a collection of non-negative functions $(\varphi_1, \dots, \varphi_m)$ such that, for any $i=1, \dots, m$, we have $\varphi_i : T \rightarrow \mathbb{R}_{\geq 0}$, $\mathrm{supp}(\varphi_i) \subset F_i$, and for any point $x \in T$, it holds that $\sum\limits_{i=1}^m \varphi_i(x) = 1$. We will only consider Hausdorff paracompact spaces $T$. For example, all manifolds and all simplicial complexes are Hausdorff paracompact and admit a partition of unity.

For each partition of unity $\Phi$, there is a continuous map
\[
c_{\Phi} : T \to |N(\mathcal{F})|, \qquad 
c_{\Phi}(p) = \sum_{i=1}^{m} \varphi_i(p) v_i.
\]
We call such a map \emph{canonical}. Any two canonical maps are homotopic. Indeed, if $\Phi$ and $\Psi$ are two partitions of unity subordinate to $\mathcal{F}$, then the homotopy between them is given by $t \Phi + (1-t)\Psi$.

Since $N(F)$ is a subcomplex of the simplex $\Delta^{m-1}$, we can consider the embedding $\alpha: N(\mathcal{F}) \rightarrow |\Delta^{m-1}|$. Suppose that the intersection of all sets is empty, $\bigcap_{i=1}^{m} F_i = \emptyset$. In other words, $N(\mathcal{F})$ does not contain an $(m-1)$-simplex. Then $\alpha: |N(\mathcal{F})| \rightarrow \partial|\Delta^{m-1}| \simeq \mathbb{S}^{m-2}$ is well-defined.

Let $|\Delta^{m-1}| = \conv(v_1, \dots, v_m)$. Let $\Phi$ be a partition of unity. Then we can define the map $\rho_{\mathcal{F},\Phi}: T \rightarrow |\Delta^{m-1}|$ as
$$\rho_{\mathcal{F},\Phi}(x) = \sum\limits_{i=1}^{m} \varphi_i(x)v_i.$$
Note that $\rho_{\mathcal{F},\Phi} = \alpha \circ c_{\Phi}$.

We now can give the following definition.

\begin{definition}[Definition 1.2 from \cite{MusH}]
\label{HomCov}
    Let $T$ be a topological space and ${\mathcal{F} = \{F_1, \dots, F_m\}}$ be an open cover such that $\bigcap_{i=1}^{m} F_i = \emptyset$. Then we denote by  $[\mathcal{F}]$ the homotopy class of the map $[\rho_{\mathcal{F},\Phi}] = [\alpha \circ c_{\Phi}]$ in the set $[T, \mathbb{S}^{m-2}]$.
\end{definition}
Since any two canonical maps are homotopic, the definition is correct and does not depend on the choice of the partition of unity $\Phi$.

\begin{remark}
    Verification that $[\mathcal{U}_L(K)] = [L]$ is straightforward.
\end{remark}

So far, we have formulated definitions only for open covers. Let us show that everything works similarly for closed covers as well.

Let $\mathcal{C} = \{C_1, \dots, C_m\}$ be a closed cover of the space $T$. We say that an open cover $\mathcal{U} = \{U_1, \dots, U_m\}$ contains $\mathcal{C}$ if $C_i \subset U_i$ for all $i = 1, \dots, m$.

\begin{lemma}[Lemma 1.4 from \cite{MusH}]
\label{ClosedLemma}
 Let $\mathcal{C} = \{C_1, \dots, C_m\}$ be a closed cover of a Hausdorff paracompact topological space $T$. Then there exists an open cover $\mathcal{U}$ of $T$ such that $\mathcal{C} \subset \mathcal{U}$ (meaning $C_i \subset U_i$ for all $i$) and $N(\mathcal{U})$ is isomorphic to $N(\mathcal{C})$.

 Moreover, if $\mathcal{U}^1$ and $\mathcal{U}^2$ are two open covers containing $\mathcal{C}$ whose nerves $N(\mathcal{U}^1)$ and $N(\mathcal{U}^2)$ are both isomorphic to $N(\mathcal{C})$, then $[\mathcal{U}^1] = [\mathcal{U}^2]$ in $[T, \mathbb{S}^{m-2}]$.
\end{lemma}


\begin{proof}
For any index set $J \subset \{1, \dots, m\}$, let $C_J := \bigcap_{j \in J} C_j$. Since $T$ is Hausdorff and paracompact, it is also a normal space. Recall that in a normal space, any two disjoint closed sets can be separated by disjoint open sets.

The proof of the first statement proceeds by induction on the dimension of the nerve, $k = \dim(N(\mathcal{C}))$.

{Base case}: if $k=0$, the sets $C_1, \dots, C_m$ are pairwise disjoint. Since $T$ is normal, we can find pairwise disjoint open sets $U_1, \dots, U_m$ such that $C_i \subset U_i$ for each $i$. The resulting open cover $\mathcal{U} = \{U_1, \dots, U_m\}$ has a nerve $N(\mathcal{U})$ that is also zero-dimensional and isomorphic to $N(\mathcal{C})$.

{Inductive Step}: assume the statement holds for all closed covers whose nerves have a dimension strictly less than $k > 0$. Let $\mathcal{J}_{max}$ be the set of all maximal simplices in $N(\mathcal{C})$. An immediate consequence of maximality is that for any two distinct maximal simplices $I, J \in \mathcal{J}_{max}$, the intersection $C_I \cap C_J$ must be empty. If it were non-empty, then $I \cup J$ would form a simplex containing both $I$ and $J$ as proper faces, contradicting their maximality. 

Therefore, the collection of closed sets $\{C_J\}_{J \in \mathcal{J}_{max}}$ is pairwise disjoint. By the normality of $T$, we can find a collection of pairwise disjoint open sets $\{W_J\}_{J \in \mathcal{J}_{max}}$ such that $C_J \subset W_J$ for each $J \in \mathcal{J}_{max}$. Furthermore, for each $i$ such that $C_i \cap C_J=\emptyset$, we can choose $W_J$ to be disjoint from an open set $\hat{U_i}$ containing $C_i$. 

Now, define a new family of sets $\mathcal{C}' = \{C'_1, \dots, C'_m\}$ by setting
$$C'_i := C_i \setminus \left( \bigcup_{J \in \mathcal{J}_{max}} W_J \right).$$
Each $C'_i$ is a closed set. For any maximal simplex $J \in \mathcal{J}_{max}$, the corresponding intersection $C'_J = \bigcap_{i \in J} C'_i$ is empty. This means no maximal simplex of $N(\mathcal{C})$ is present in $N(\mathcal{C}')$, so $\dim(N(\mathcal{C}')) < \dim(N(\mathcal{C}))$.

By the induction hypothesis, there exists an open cover $\mathcal{U}' = \{U'_1, \dots, U'_m\}$ such that $C'_i \subset U'_i$ for all $i$ and $N(\mathcal{U}') \cong N(\mathcal{C}')$.

Next, we must adjust the sets from $\mathcal{U}'$. For a given $i$, consider the intersection $\left(U_i' \cap \hat{U_i}\right)$. By a slight abuse of notation, we continue to call this refined set $U'_i$. This adjustment guarantees that for any maximal simplex $J$, we have $U'_i \cap W_J = \emptyset$ if $i \notin J$. And clearly $C_i' \subset U_i'$.

Finally, we construct our desired open cover $\mathcal{U}=\{U_1, \dots, U_m\}$ by defining
$$U_i := U'_i \cup \left( \bigcup_{J \in \mathcal{J}_{max}, \, i \in J} W_J \right).$$
By construction, $C_i \subset U_i$ for all $i$, and the nerve $N(\mathcal{U})$ is isomorphic to $N(\mathcal{C})$. This completes the induction.

\smallskip

We now prove the second part of the lemma. Suppose there exist two open covers $\mathcal{U}^1 = \{U^1_1, \dots, U^1_m\}$ and $\mathcal{U}^2 = \{U^2_1, \dots, U^2_m\}$ such that $\mathcal{C} \subset \mathcal{U}^1$, $\mathcal{C} \subset \mathcal{U}^2$, and $N(\mathcal{U}^1) \cong N(\mathcal{U}^2) \cong N(\mathcal{C})$.

Consider the cover $\mathcal{U}' = \{U'_1, \dots, U'_m\}$ where $U'_i := U^1_i \cap U^2_i$. This is an open cover containing $\mathcal{C}$ whose nerve $N(\mathcal{U}')$ is isomorphic to $N(\mathcal{C})$. The cover $\mathcal{U}'$ is a refinement of both $\mathcal{U}^1$ and $\mathcal{U}^2$ and $\mathcal{U'}$ contains $\mathcal{C}$. Hence their nerves are isomorphic and they define the same homotopy class: $[\mathcal{U}^1] = [\mathcal{U}']=[\mathcal{U}^2]$. 
\end{proof}

This lemma shows that we can define the homotopy class $[\mathcal{C}]$ of the closed cover $\mathcal{C}$ as the class $[\mathcal{U}]$, where $\mathcal{U}$ contains $\mathcal{C}$ and the nerve $N(\mathcal{U}) $ is isomorphic to the nerve $N(\mathcal{C})$.

In \cite{MusH} the following theorem was proven.

\begin{theorem}[Theorem 1.1 from  \cite{MusH}]
\label{Exist}
    Let $T$ be a topological space, and let $h \in [T, \mathbb{S}^{m-2}]$. Then there exists an open cover $\mathcal{F} = \{F_1, \dots, F_m\}$ such that $[\mathcal{F}] = h$.

    If $T$ is a simplicial complex, then there exists a triangulation $K$ of $T$ and a coloring $L : \mathrm{Vert}(K) \to [m]$ such that $[L] = h$.
\end{theorem}

In the paper \cite{MusH}, theorems on cover extension are also proved. In Section 3, we will present proofs of generalized versions of these theorems.

\begin{definition}[Definition 2.1 from \cite{MusH}]
\label{ExtDef}
    Let $A$ be a subspace of a topological space $X$. Suppose $\mathcal{S} = \{S_1, \dots, S_m\}$ is an open (or closed) cover of $A$, and $\mathcal{F} = \{F_1, \dots, F_m\}$ is an open (or, respectively, closed) cover of $X$. We say that $\mathcal{F}$ is an \textit{extension} of $\mathcal{S}$ if $S_i = F_i \cap A$ for all $i$.
\end{definition}

In the case $A = S^{k}$ and $X = B^{k+1}$, the following theorem holds.

\begin{theorem}[Theorem 2.1 from  \cite{MusH}]
\label{ExtSph}
    Let $\mathcal{S} = \{S_1, \dots, S_m\}$ be a cover of $\mathbb{S}^k$ such that the intersection of all $S_i$ is empty. Then $\mathcal{S}$ can be extended to a cover $\mathcal{F}$ of the ball $\mathbb{B}^{k+1}$, such that the intersection of all $F_i$ is empty if and only if $[\mathcal{S}] = 0$ in $\pi_k(\mathbb{S}^{m-2})$.
\end{theorem}

Now, let us consider the case where $A$ is the boundary of a manifold $M$.

\begin{definition}[Definition 2.2 from \cite{MusH}]
    Let $N$ be an $(n-2)$-dimensional manifold without boundary, and suppose $\mathcal{S} = \{S_1, \dots, S_n\}$ is a cover of $N$. If the intersection of all $S_i$ is empty, then $[\mathcal{S}] \in \mathbb{Z} = [N, \mathbb{S}^{n-2}]$. The value $[\mathcal{S}]$ is called the \textit{degree} of $\mathcal{S}$, denoted as $\deg(\mathcal{S})$.
\end{definition}

\begin{theorem}[Theorem 2.2 from  \cite{MusH}]
\label{ExtMan}
    Let $M$ be an $(n-1)$-dimensional manifold, and let $N = \partial M$. Suppose $\mathcal{F} = \{S_1, \dots, S_n\}$ is a cover of $N$ such that the intersection of all $S_i$ is empty. Then $\mathcal{F}$ can be extended to a cover $\mathcal{F} = \{F_1, \dots, F_n\}$ of the manifold $M$, such that the intersection of all $F_i$ is empty if and only if $\deg(\mathcal{S}) = 0$.
\end{theorem}

\section{Balanced subsets}

In this Section we explore some properties of balanced subsets of points and prove Theorem \ref{NonBalComp1}

Let $V=\{v_1, \dots, v_m\} \subset \mathbb{R}^{d}$ be a set of points in Euclidean space $\mathbb{R}^d$. Denote by $rk(V)$ the dimension of the polytope $\conv(V)$. The set $V$ is of full rank if $rk(V)=d$.

\begin{definition}
    
Suppose there is a point $r \in \conv(V)$. A set of vertices $S \subset V$ is called $r$-balanced if $r \in \conv(S)$.  Denote the family of all $r$-balanced subsets as $\bs(V,r)$. The corresponding set of indices is also called $r$-balanced. \end{definition}

To a family of points $V$ and to a point $r$ we will refer simply as a pair $(V,r)$ in $\mathbb{R}^d$. Denote by $[(V,r)]$ a family of all pairs $(V',r')$ that are  $\bs$-equivalent to $(V,r)$.

\medskip

Note that moving point $v_i$ along the ray from $r$ to $v_i$ does not change the family of balanced sets. This can be formalized and proved in the next obvious technical lemma. 

\begin{lemma}
  \label{conveq}  
 Let $(V,r)$ be a pair where $V =\{v_1, \dots, v_m\} \subset \mathbb{R}^d$ and $r \in \mathbb{R}^d$, let $\{\lambda_i\}_{i=1}^{m}$ be a set of positive numbers. Then $(V,r)$ is $\bs$-equivalent to $(V',r)$, where $V'=\{v_1', \dots, v_m'\}$ and $v_i'=\lambda_i(v_i-r)+r$ for $i=1, \dots, m$.
\end{lemma}
\begin{proof}
    Suppose $S =\{v_{i_1}, \dots, v_{i_k}\} \in \bs(V)$ and $S'=\{v_{i_1}', \dots, v_{i_k}'\}$. Then we have $r=\sum \limits_{j=1}^{k}c_jv_{i_j}$ where $\sum \limits_{j=1}^{k}c_j=1$ and $c_j \geq 0$ for $j=1, \dots, k$. 
    
    It remains to show that $r \in \conv(S')$. Indeed, consider $c=\sum\limits_{j=1}^{k}\frac{c_j}{\lambda_{i_j}}$ and $c_j'=\frac{c_j}{\lambda_{i_j}c}$.  

   It is clear that $\sum \limits_{j=1}^{k}c_j'=1$, $c_j' \geq 0$, and $r=\sum \limits_{j=1}^{k}c_j'v_{i_j}'$
\end{proof}

Now, consider the transformation 
$$
\tau_{r}(x) =
\begin{cases} 
      \frac{x - r}{\|x - r\|} & \text{if } x \neq r, \\
      r & \text{if } x = r.
\end{cases}
$$  

From \ref{conveq} it follows that

\begin{lemma}
 \label{RemBall}
  $\tau_{r}$ does not change $r$-balanced sets. That is, if $S \subset V$ is a set of points, then $r \in \conv(S)$ if and only if $r \in \conv(\tau_{r}(S))$.
\end{lemma}

\medskip

Let us now examine in more detail the sets that are not balanced. For a pair $(V,r)$, consider the family of convex sets $\mathcal{P}(V,r)$, where $F \in \mathcal{P}(V,r)$ if and only if $F = \conv(S)$ for some $S \notin \bs(V,r)$. We denote their union by 
\[
P(V,r) = \bigcup \mathcal{P}(V,r).
\] 

We first establish the homotopy type of $P(V,r)$. 
\begin{lemma}
\label{PBallLemma}   Let $(V,r)$ be a pair in $\mathbb{R}^d$ of rank $k \leq d$. Then 
    \begin{enumerate}
  
        \item If  $r \in \mathrm{relint}(\conv(V))$, then  $P(V,r) \simeq S^{k-1}$ and $\tau_r: P(V,r) \rightarrow S^{k-1} \subset S^{d-1}$ is a homotopy equivalence.
        \item If $r \notin \mathrm{relint}(\conv(V))$, then $P(V,r)$ is contractible. 
    \end{enumerate}
\end{lemma}

\begin{proof}
In the remainder of the proof we assume $r=0$. We first reduce the case of rank $k$ to the case of full rank. Note that if $0 \notin \mathrm{Aff}(V)$, where $\mathrm{Aff}(V)$ denotes the affine hull of $V$, then all subsets of $V$ are non-balanced and $P(V,r)=\conv(V)$ is contractible. Otherwise, $(V,r)$ is a pair of full rank in the corresponding affine subspace.

    For the first part, consider a ray $l(v) = \{ tv \mid v \in S^{d-1},\; t \geq 0 \}$ starting at $0$ in the direction $v$.
    
    \begin{claim}
    \label{convexityclaim}
        The intersection $l(v) \cap P(V,0)$ consists of a single connected component of the form $[t_1(v)v,\, t_2(v)v] \subset l(v)$ with $t_1(v) \leq t_2(v)$. 
    \end{claim}

    \begin{proof}
        Suppose instead that $l(v) \cap P(V,0)$ contains several segments. Let $[t_1v, t_2v]$ be the first such segment and $[t_3v, t_4v]$ the second, with $(t_2v, t_3v) \cap P(V,0) = \emptyset$. Let $S_1 \subset V$ be the maximal subset such that $0 \notin \conv(S_1)$ and $\conv(S_1) \cap l = t_2v$. Since $t_2v \in P(V,0)$, such an $S_1$ exists. Similarly define $S_2$ for the point $t_3v$.  

        Now consider the hyperplane $H$ through $t_2v$ that separates $S_1$ from $t_3v$. Clearly $\conv(S_1) \subset H$, and $H$ separates $0$ from $t_3v$. Assume $t_3v$ lies in the positive half-space $H^+$. Then at least one point $v_i \in S_2$ also lies in $H^+$. Consider the set $S_1 \cup \{v_i\}$. Since all points of $S_1 \cup \{v_i\}$ avoid the negative half-space $H^-$, we have $0 \notin \conv(S_1 \cup \{v_i\})$. If $\conv(S_1 \cup \{v_i\}) \cap l = t_2v$, then $S_1$ was not maximal. Otherwise $\conv(S_1 \cup \{v_i\}) \cap l$ is a segment $[t_2v, av]$ with $a > t_2$, which contradicts the assumption.  
    \end{proof}   

    This claim is the key step to prove homotopy equivalence. For any $x \in P(V,0)$ there exists a unique $v(x) \in S^{d-1}$ and $t(x) > 0$ such that $x = t(x)v(x)$; moreover, the functions $v(x)$ and $t(x)$ are continuous. Since $0 \in \mathrm{int}(\conv(V))$, the set $\{t_2(v)v \mid v \in S^{d-1}\}$ is exactly $\partial \conv(V)$. Thus $\hat{l}(v) = t_2(v)v$ gives a homeomorphism between $S^{d-1}$ and $\partial \conv(V)$. 

    We now show that $\partial \conv(V)$ is a deformation retract of $P(V,0)$. Define a homotopy $f_h: P(V,0) \to P(V,0)$ by 
    \[
    f_h(x) = (1-h)\,t(x)v(x) + h\,\hat{l}(v(x)).
    \] 
    By Claim~\ref{convexityclaim}, $f_h$ is a deformation retraction. Note that
    \[
    \tau_{0}(P(V,0)) = S^{d-1}, \qquad 
    \tau_{0} \circ \hat{l} = \mathrm{id}_{S^{d-1}}, \qquad 
    \hat{l} \circ \tau_{0} \simeq \mathrm{id}_{P(V,0)},
    \]
    hence $\tau_{0}$ is a homotopy equivalence.

    \medskip
  For the second part, suppose $0 \notin \conv(V)$. Then $P(V,0) = \conv(V) \simeq B^{d}$. If $0 \in \partial(\conv(V))$, then there exists a supporting hyperplane $W$ with $0 \in W$. This hyperplane divides $\mathbb{R}^d$ into a closed half-space $\mathrm{Cl}({W}^{-})$ and an open half-space $W^{+}_{}$, and it also partitions $V$ into two subsets:
\[
S_1 = \{ v_i \in V \mid v_i \in W \}, \qquad 
S_2 = \{ v_i \in V \mid v_i \in W^{+} \}.
\] 
Clearly $S_2 \neq \emptyset$. Assume $S_1$ has rank $k \leq d-1$. Then $0$ lies in the relative interior of $\conv(S_1)$, and $P(S_1,0) \simeq S^{k-1}$ from the previous part of the lemma. For $S_2$ we have $P(S_2,0) = \conv(S_2) \simeq B^{\ell}$ for some $\ell$. 

Now pick any point $v_i \in S_2$ (such a point exists since $S_2 \neq \emptyset$). For any subset $S \subset V$, if $0 \notin \conv(S)$ then also $0 \notin \conv(S \cup \{v_i\})$. Hence $P(V,0)$ is a cone with apex $v_i$, and therefore $P(V,0)$ is contractible.
\end{proof}

The family of non-balanced sets forms an abstract simplicial complex $\mathcal{K}(V,r)$, whose vertex set is
\[
V(\mathcal{K}(V,r)) = [m],
\]
and where $I \subset [m]$ is a simplex if and only if $r \notin \conv\{v_i \mid i \in I\}$.

The next theorem describes the homotopy type of $\mathcal{K}(V,r)$.

\begin{theorem}
\label{NonBalComp}
   Let $(V,r)$ be a pair in $\mathbb{R}^d$ of rank $k \leq d$. Then 
   \begin{enumerate}
       \item If $r \in \mathrm{relint}(\conv(V))$, then $|\mathcal{K}(V,r)|$ is homotopy equivalent to the sphere $S^{k-1}$.
       \item If $r \notin \mathrm{relint}(\conv(V))$, then $|\mathcal{K}(V,r)|$ is contractible. 
     
   \end{enumerate}

\end{theorem}

We remark that Theorem~\ref{NonBalComp1} is precisely the first part of Theorem~\ref{NonBalComp}. We now turn to the proof.

\begin{proof}

As in the proof of Lemma~\ref{PBallLemma}, the same argument reduces our theorem to the full-rank case.

For $\bs$-equivalent pairs we obtain the same non-balanced simplicial complexes. By Lemma~\ref{RemBall} it is therefore sufficient to consider the case $V \subset S^{d-1}$ and $r=0$.

We begin with the first part of the theorem. With each point $v_i$ we associate an open hemisphere
\[
H_{v_i} = \{ x \in S^{d-1} \mid \langle x, v_i \rangle > 0 \}.
\]

\begin{enumerate}
    \item The family $\{H_{v_i}\}_{i=1}^m$ forms an open cover of $S^{d-1}$. Indeed, suppose there exists $y \in S^{d-1}$ such that $y \notin H_{v_i}$ for all $i=1,\dots,m$. Then $\langle y, v_i \rangle \leq 0$ for all $i$, so all points $v_i$ lie in the closed hemisphere opposite to $y$. This would imply $0 \notin \mathrm{int}(\conv(V))$, a contradiction. Thus $\mathcal{H}=\{H_{v_i}\}_{i=1}^m$ is an open cover of $S^{d-1}$.
    
    \item The nerve of this cover, $N(\mathcal{H})$, is isomorphic to $\mathcal{K}(V,0)$. Indeed, if $0 \notin \conv(S)$, then there exists a hyperplane through $0$ separating $0$ from $S$. Choosing $y \in S^{d-1}$ orthogonal to this hyperplane, we have $\langle y, v_i \rangle > 0$ for each $v_i \in S$, so $\bigcap_{v_i \in S} H_{v_i} \neq \emptyset$. Hence $\mathcal{K}(V,0) \subset N(\mathcal{H})$. Conversely, if $\bigcap_{v_i \in S} H_{v_i} \neq \emptyset$, then there exists $y \in S^{d-1}$ with $\langle y, v_i \rangle > 0$ for all $v_i \in S$, which implies $0 \notin \conv(S)$. Thus $N(\mathcal{H}) \cong \mathcal{K}(V,0)$.
    
    \item Since each $H_{v_i}$ is convex, any finite intersection $\bigcap_{v_i \in S} H_{v_i}$ is contractible. Therefore $\mathcal{H}$ is a good cover. By the Nerve Theorem (see, for example, \cite{Hatcher}), $|\mathcal{K}(V,0)|$ is homotopy equivalent to $S^{d-1}$.
\end{enumerate}

We now turn to the second part. If $0 \notin \conv(V)$, then $\mathcal{K}(V,0) = 2^{[m]}$ is the full simplex and therefore contractible. If $0 \in \partial(\conv(V))$, then there exists a supporting hyperplane $W$ with $0 \in W$. This hyperplane divides $\mathbb{R}^d$ into a closed half-space $\mathrm{Cl}({W}^{-})$ and an open half-space $W^{+}_{}$, and it also partitions $V$ into two subsets:
\[
S_1 = \{ v_i \in V \mid v_i \in W \}, \qquad 
S_2 = \{ v_i \in V \mid v_i \in W^{+} \}.
\] 
Clearly $S_2 \neq \emptyset$. Suppose $S_1$ has rank $l_1 \leq d-1$; then $\mathcal{K}(S_1,r) \simeq S^{l_1}$. For $S_2$ we have $\mathcal{K}(S_2,r) \simeq \Delta^{l_2}$ for some $l_2$.Using the same argument as in Lemma~\ref{PBallLemma}, we see that $\mathcal{K}(V,0)$ is a cone over any point $v_i \in S_2$, and hence contractible. In fact, $\mathcal{K}(V,r)$ can be written explicitly as the join
\[
\mathcal{K}(V,r) = \mathcal{K}(S_1,r) * \mathcal{K}(S_2,r).
\]

\end{proof}

Note that Lemma~\ref{PBallLemma} and Theorem~\ref{NonBalComp} can be applied to $(V,r)$ even if it is not of full rank. Indeed, consider the minimal affine subspace containing $\conv(V)$; within this subspace the pair $(V,r)$ is of full rank.

\begin{corollary}
    Let $(V,r)$ be a pair of rank $n$ with $r \in \mathrm{int}(\conv(V))$, and let $(V',r')$ be a pair of rank $k$ with $r' \in \mathrm{int}(\conv(V'))$. Suppose $(V,r)$ is $\bs$-equivalent to $(V',r')$. Then $n=k$.
\end{corollary}

\begin{proof}
    We have 
    \[
    S^{n-1} \simeq |\mathcal{K}(V,r)| = |\mathcal{K}(V',r')| \simeq S^{k-1}.
    \] 
\end{proof}

\medskip

It is clear that $P(V,r)$ is the image of $|\mathcal{K}(V,r)|$ under the affine map 
\[
\pi: |\mathcal{K}(V,r)| \longrightarrow \mathbb{R}^d, \qquad \pi(i) = v_i.
\] 
In fact, we have the following:

\begin{lemma}
\label{PiHomEq}
  The map $\pi: |\mathcal{K}(V,r)| \to P(V,r)$ is a homotopy equivalence.   
\end{lemma}

\begin{proof}
As usual, we restrict to the case $V \subset S^{d-1}$.
\begin{enumerate}
    \item Let $\mathcal{H}$ be the cover of $S^{d-1}$ defined in the proof of Theorem~\ref{NonBalComp}. Then there exists a partition of unity $\Phi_{\mathcal{H}} = \{\varphi_1, \dots, \varphi_{m}\}$ and a map $\rho_{\mathcal{H}}: S^{d-1} \to |\mathcal{K}(V,r)|$. By the Nerve Theorem, $\rho_{\mathcal{H}}$ is a homotopy equivalence, and moreover $\deg(\rho_{\mathcal{H}}) = \pm 1$. 
    \item The map $\tau_{r}: P(V,r) \to S^{d-1}$ is a homotopy equivalence with $\deg(\tau_{r}) = 1$.
    \item Define $f: S^{d-1} \to S^{d-1}$ by 
    \[
    f = \tau_{r} \circ \pi \circ \rho_{\mathcal{H}}, 
    \qquad 
    f(x) = \frac{\sum\limits_{i=1}^m \varphi_i(x) v_i}{\left\lVert \sum\limits_{i=1}^m \varphi_i(x) v_i \right\rVert}.
    \] 
    Note that if $\varphi_i(x) > 0$ then $\langle x, v_i \rangle > 0$, hence also $\langle x, f(x) \rangle > 0$. Thus $f(x) \neq -x$ for any $x \in S^{d-1}$. For such maps it is known that $\deg(f) = \pm 1$. Consequently $\deg(\pi) = \pm 1$, and therefore $\pi$ is a homotopy equivalence.
\end{enumerate}
\end{proof}

\section{Balanced subsets and covers}

In the work \cite{MusH} a generalization of the Definition \ref{HomCov} was proposed. In this generalization instead of vertices of the simplex $\Delta^{m-1}$ we consider an arbitrary set of points $V=\{v_1,\dots,v_m\}\subset \mathbb{R}^d$. In this Section we will take a closer look at this object. In particular, we present modified versions of the Theorems \ref{ExtSph} and \ref{ExtMan}.

\begin{definition}
    Suppose we have a pair $(V,r)$ in $\mathbb{R}^d$ with $V=\{v_1, \dots, v_m\}$. Let ${\mathcal{F}=\{F_1,\dots,F_m\}}$ be a cover of a topological space  $T$. For a simplex $\sigma \in N(\mathcal{F})$ we say that $\sigma$ is $r$-balanced (with respect to $(V,r)$) if $\sigma$ corresponds to $r$-balanced subset of $V$.
\end{definition}

For an open cover $\mathcal{F}$ of $T$ we can choose a partition of unity $\Phi=\{\varphi_1, \dots, \varphi_m\}$ and define a map $\rho_{\Phi,V}: T \rightarrow \conv(V)$ where $$\rho_{\Phi,V}(x)=\sum\limits_{i=1}^{m}\varphi_i(x)v_i.$$ Suppose there are no $r$-balanced simplices in  $N(\mathcal{F})$. Then for any partition of unity $\Phi$ the map $\rho_{\Phi,V}$ will be a map from $T$ to $P(V,r)$. In particular, the point $r$ will not be an image of the map $\rho_{\Phi,V}$. Then the map $g_{\Phi,V,r}: T \rightarrow S^{d-1}$, $g_{\Phi,V,r}=\tau_{r}\circ \rho_{\Phi,V}$ is a correct continuous map. Moreover, in \cite{MusH} the following lemma was proven.

\begin{lemma}
    For any two partitions of unity the corresponding maps $g_{\Phi,V,r}$ and $g_{\Psi,V,r}$ are homotopic and they define the same homotopy class in $[T, \mathbb{S}^{d-1}]$.
\end{lemma}

\begin{proof}
    Let there be two partitions of unity $\Phi$ and $\Psi$. Consider a linear homotopy ${\Phi_t=(1-t)\Phi+t\Psi}$. $\Phi_t$ is a partition of unity and $g_{\Phi_t,V,r}$ is a correct mapping for any $t$, then $g_{\Phi_t,V,r}$ is a homotopy between $g_{\Phi,V,r}$ and $g_{\Psi,V,r}$.
\end{proof}

Actually, if $N(\mathcal{F})$ does not have $r$-balanced simplices, then there is an embedding $\iota : |N(\mathcal{F})| \hookrightarrow |\mathcal{K}(V,r)|$. Note that the partition of unity defines a canonical map $c_{\Phi}: T \rightarrow N(\mathcal{F})$ and $P(V,r)=\pi(|\mathcal{K}(V,r)|)$. Then the map $\rho_{\Phi,V}$ can be decomposed into the composition $\rho_{\Phi,V}=\pi \circ \iota \circ c_{\Phi} $ and $g_{\Phi,V,r}=\tau \circ \pi \circ \iota \circ c_{\Phi}$. 

Now  we are ready to give the following definition.

\begin{definition}
  \label{VHomCov}
  Let $(V,r)$ be a pair in $\mathbb{R}^d$. Let $\mathcal{F}=\{F_1, \dots, F_m\}$ be an open cover of $T$ such that there is no $r$-balanced simplices in the nerve $N(\mathcal{F})$. Denote by $[\mathcal{F}_{(V,r)}]$ the homotopy class of the map $[g_{\Phi,V,r}]=[\tau \circ \pi \circ \iota \circ c_{\Phi}]$ in the set $[T, \mathbb{S}^{d-1}]$. In the case when there exists $r$-balanced simplex in $N(\mathcal{F})$ then this homotopy class is undefined and we denote this as $[\mathcal{F}_{(V,r)}] = N/A$

  If $T$ is an oriented manifold of dimension $d-1$ then the degree $\text{deg}(\mathcal{F}_{(V,r)})$ is well defined.
\end{definition}

It is clear that for a closed cover we can define a homotopy class (relative to $(V,r)$) in the same way as in Lemma \ref{ClosedLemma}. It is also clear that, similarly, one can define the homotopy class $[L_{(V, r)}]$ of a coloring $L: Vert(K) \rightarrow [m]$ of the vertices of a simplicial complex $K$ as the class of a piecewise-linear mapping $f_L: |K| \rightarrow P(V,r)$. And if there exists $r$-balanced simplex in $K$ this is denoted as $[L_{(V,r)}]=N/A$. 

\begin{remark}
    Note that, as before, we can consider only the case when $V \subset S^{d-1}$ since $\tau(V)$ is $\bs$-equivalent to $V$ and $g_{\Phi,V,r}$ is the same map as $g_{\Phi,\tau_{r}(V),r}$.
\end{remark}

Definition~\ref{VHomCov} depends strongly on the choice of the pair $(V,r)$. It is clear that the homotopy class of a cover $[\mathcal{F}_{(V,r)}]$ may change if we choose another pair $(V',r')$. In fact, it is easy to construct an example where $[\mathcal{F}_{(V,r)}] = 0$ while $[\mathcal{F}_{(V',r')}] \neq 0$ if $(V,r)$ and $(V',r')$ are not $\bs$-equivalent. Moreover, this homotopy class can also depend on the enumeration of the points $V = \{v_1, \dots, v_m\}$. 

What happens if $(V,r)$ is $\bs$-equivalent to $(V',r')$? The next theorem answers this question. Before stating it, we need to define an involution $\mathrm{inv}$ on the set $[T,S^{d-1}]$.

If $f: S^{d-1} \rightarrow S^{d-1}$ is a homotopy equivalence then $\text{deg}(f)=\pm1$. Note that $[f]$ induces an automorphism on $[T, S^{d-1}]$ by  $[f]: [g] \mapsto [f \circ g]$. If $\text{deg}(f)=1$ then $[f]: [T, S^{d-1}] \rightarrow [T, S^{d-1}]$ is the identity map. If $\text{deg}(f)=-1$ then $[f]: [T, S^{d-1}] \rightarrow [T, S^{d-1}]$ is an involution and $[f]^2=id$. This involution we denote simply  as $\text{inv}$. Note that $[g]=0$ in $[T, S^{d-1}]$ if and only if $\text{inv}([g])=0$ in $[T, S^{d-1}]$.

\begin{theorem}

\label{IndTh}
    Suppose we have a pair $(V,r)$ and a pair $(V',r')$ in $\mathbb{R}^d$ such that $(V,r)$ is $\bs$-equivalent to $(V',r')$. Let $\mathcal{F}$ be a cover of $T$. Then $[\mathcal{F}_{(V,r)}]=[\mathcal{F}_{(V',r')}]$ or $[\mathcal{F}_{(V,r)}]=\text{inv}([\mathcal{F}_{(V',r')}])$. In particular, $[\mathcal{F}_{(V,r)}]=0 \iff[\mathcal{F}_{(V',r')}]=0$.

    If $T$ is an oriented manifold of dimension $d-1$ then $\text{deg}(\mathcal{F}_{(V,r)})=\pm \text{deg}(\mathcal{F}_{(V',r')})$
\end{theorem}

\begin{proof}
    
From the decomposition $g_{\Phi,V,r}=\tau \circ\pi\circ\iota \circ c_{\Phi}$ and from  Lemma \ref{PBallLemma}, Theorem \ref{NonBalComp}, and Lemma \ref{PiHomEq}  we have three cases
\begin{enumerate}
    \item $[g_{\Phi,V,r}]=[\iota \circ c_{\Phi}]$ if $r \in \text{int}(\conv(V))$ and $\text{deg}(\pi)=1$.
    \item $[g_{\Phi,V,r}]=\text{inv}([\iota \circ c_{\Phi}])$ if $r \in \text{int}(\conv(V))$ and $\text{deg}(\pi)=-1$.
    \item $[g_{\Phi,V,r}]=0=[\iota \circ c_{\Phi}]$ if $r \notin \text{int}(\conv(V))$.
\end{enumerate}
The rest follows from the fact that $\mathcal{K}(V,r)=\mathcal{K}(V',r')$.
\end{proof}

Now we will extend Theorem \ref{Exist} and Theorem \ref{ExtSph} to the $r$-balanced case.

\begin{theorem}
     \label{BalExist}

     Let $T$ be a topological space and $h \in [T,\mathbb{S}^{d-1}]$ be a homotopy class. Suppose we have a pair $(V,r)$ in $\mathbb{R}^d$ of full rank such that $r \in \text{int}(\conv(V))$. Then there exists an open cover $\mathcal{F}=\{F_1, \dots, F_m\}$ of $T$ such that $[\mathcal{F}_{(V,r)}]=h$.
\end{theorem}

\begin{proof}
  From Theorem \ref{IndTh} we know that $[\mathcal{F}_{(V,r)}]=[\iota \circ c_{\Phi}]$ (up to involution) and from Theorem \ref{NonBalComp} we know that $\mathcal{K}(V,r) \simeq S^{d-1}$. So it is sufficient to prove the theorem by constructing a cover $\mathcal{F}$ such that $[\iota \circ c_{\Phi}]=h \in [T, |\mathcal{K}(V,r)|]$ for some partition of unity $\Phi$. 

  Fix a coloring $ L(v_i) = i $ for $ v_i \in Vert(\mathcal{K}(V,r)) $ and consider the open cover $\mathcal{U}_L(\mathcal{K}(V,r))$ of the $|\mathcal{K}(V,r)|$ where $ \mathcal{U}_L(\mathcal{K}(V,r)) = \{ U_i(\mathcal{K}(V,r)) \}_{i=1}^{m} $. Note that the nerve of this cover is exactly $\mathcal{K}_{(V,r)}$; the homotopy class of this cover relative to $(V,r)$ corresponds to mapping degree $\pm1$ and for a partition of unity $\Psi$ subordinate to this cover we have $[\iota \circ c_\Psi ]=1 \in [|\mathcal{K}(V,r)|, |\mathcal{K}(V,r)|]$. Moreover,we can choose $\Psi$ such that $c_{\Psi}=id$.

Let $f: T \rightarrow |\mathcal{K}(V,r)|$ be a map such that $[f]=h$ in $[T,|\mathcal{K}(V,r)|]$. Define a cover of the space $T$ as $F_i={f}^{-1}(U_i(\mathcal{K}(V,r)))$. This cover is exactly what we need: the homotopy class of $\mathcal{F}=\{F_1, \dots, F_m\}$ relative to $(V,r)$ is exactly $[h]$. To see this, note that $\Phi=\Psi \circ \hat{f}$ is a partition of unity subordinate to $\mathcal{F}$. It is clear then that $[\mathcal{F}_{(V,r)}]=[\iota \circ c_{\Phi}]=[f]=h$. 

If $T$ is a simplicial complex, then by the simplicial approximation theorem we choose a triangulation $K$ of the space $T$ and a piecewise-linear map $g: K \rightarrow |\mathcal{K}(V,r)|$ such that $g$ is homotopic to $\hat{f}$ and $[g]=h$. For all $u \in Vert(K)$, we define $L(u)=i$ if $g(u)=v_i$. Then $[L_{(V,r)}]=h$ (up to involution).
  
\end{proof}

Now we have made the necessary preparations to prove Theorem \ref{VExtSph}.
\begin{proof}
The proof largely repeats the proof of Theorem \ref{ExtSph} from \cite{MusH}, but contains significant additions.

If $\mathcal{S}$ can be extended to such $\mathcal{F}$, then we have $g_{S,V,r}: S^{k} \rightarrow S^{d-1}$ and $g_{F,V,r}: B^{k+1} \rightarrow S^{d-1}$. Since $g_{S,V,r} = g_{F,V,r} \circ in$, where $in : {S}^{k} \rightarrow {B}^{k+1}$ is a homotopically trivial map, we have $[\mathcal{S}_{(V,r)}] = [g_{S,V,r}] = 0$.

    Now we will show that if $[\mathcal{S}_{(V,r)}]=0$ then $\mathcal{S}$ can be extended to $\mathcal{F}$. From Lemma \ref{ClosedLemma} it is clear that we can consider only the case of open covers. Choose a partition of unity $\Phi=\{\varphi_1, \dots, \varphi_m\}$. Note that if $[\mathcal{S}_{(V,r)}]=0$ then $[\iota \circ c_{\Phi}]=0$ in $[S^{k}, |\mathcal{K}(V,r)|]$. Then we have a homotopy: $$ \tilde{H}: {S}^{k} \times [0,1] \rightarrow |\mathcal{K}(V,r)| $$ where $ \tilde{H}(x,0) =\iota \circ c_{\Phi}$ and $ \tilde{H}(x,1) = v_1 $ for some vertex $v_1$ from $\mathcal{K}(V,r)$. As before, consider the cover $ \mathcal{U}_L(\mathcal{K}(V,r)) = \{ U_i(\mathcal{K}(V,r)) \}_{i=1}^{m} $. Let $ D = {S}^{k} \times [0,1] $. Then we define 
$ U_{i}(\Phi,D) = \tilde{H}^{-1}(U_i(\mathcal{K}(V,r))), ~~~~ \mathcal{U}_L(\Phi,D) = \{ U_i(\Phi,D) \}_{i=1}^{m}. $

It is clear that $ \mathcal{U}_L(\Phi,D) $ is an open cover of the space $ D $. However, it is not an extension of the cover $ \mathcal{S} $, since the restriction $  \mathcal{U}_L(\Phi,D)|_{S^k \times \{0\}} $ consists of sets 
$$ U_i(\Phi,D)|_{{S}^{k}\times \{0\}} = \{ x \in {S}^{k} : \phi_i(x) > 0 \} \subset S_i. $$

It remains only to extend $ U_i(\Phi,D)|_{{S}^{k}\times \{0\}} $ to $ S_i $. We will do this as follows: let $ \Pi(\mathcal{S}) $ be the set of all partitions of unity subordinate to $ \mathcal{S} $. Then 
$ S_i = \bigcup_{\Phi \in \Pi(\mathcal{S})} U_i(\Phi,D)|_{{S}^{k} \times \{0\}}. $

We denote 
$$ W_i = \bigcup_{\Phi \in \Pi(\mathcal{S})} U_i(\Phi,D). $$

Then $ \mathcal{W} = { W_1, \dots, W_m } $ will be a cover of $ D $ that extends $ \mathcal{S} $. The boundary of $ D $ consists of two spheres, $ D_0 = {S}^{k} \times \{ 0 \} $ and $ D_1 = {S}^{k} \times \{ 1 \} $. The cover $ \mathcal{W} $ on $ D_0 $ coincides with $ \mathcal{S} $, while on $ D_1 $, it is covered by a single set $ W_1 $. Let $ Z $ be a $ k+1 $-dimensional disk with boundary $ D_1 $, and let $ F_1 = W_1 \cup Z $. We consider $ B = D \cup Z $. It is clear that this is a $ k+1 $-dimensional disk. We also note that $ \mathcal{F} = { F_1, W_2, \dots, W_m} $ is a cover of $ B $ that extends $ \mathcal{S} $ and such that $ N(\mathcal{F}) $ does not contain $ r $-balanced simplices.

\end{proof}

Here we proof the mapping degree version of the Theorem \ref{VExtMan}.

\begin{proof}
    From the Hopf degree theorem a map $f: \partial M \rightarrow \mathbb{S}^{d-1}$ can be extended to $F: M \rightarrow \mathbb{S}^{d-1}$ with $F|_{\partial M}=f$ if and only if $deg(f)=0$.Then if the cover $\mathcal{S}$ can be extended to such a cover $\mathcal{F}$ of $M$ that there is no $r$-balanced simplices in $N(\mathcal{F})$ then $deg(\mathcal{S}_{(V,r)})=0$. 

    \medskip

    If $deg(\mathcal{S}_{(V,r)})=0$ then the argumant is the same as in the proof of Theorem \ref{VExtSph}. From the Collar theorem there is an $\varepsilon$ neighborhood $C\subset M$ of $\partial M$ such that $C$ is homeomorphic to   $\partial M \times [0,1]=D$. Let  $F_1=W_1 \cup (M \setminus D)$ and $\mathcal{F}=\{F_1, W_2, \dots,  W_m\}$ is the desired cover of $M$.
\end{proof}

\section{Covers of Euclidean space}

Suppose we have a closed (or open) cover $\mathcal{F}=\{F_1, \dots, F_m\}$ of $\mathbb{R}^n$. Let ${f: B^{k} \rightarrow \mathbb{R}^n}$ be a map from a ball to Euclidean space. Define the induced cover $f^*\mathcal{F}$ on $B^{k}$ as ${x \in f^*F_i \iff f(x) \in F_i}$. 

\begin{definition}
     Let  $\mathcal{F}=\{F_1, \dots, F_m\}$ be a closed (or open) cover of $\mathbb{R}^n$, let ${V=\{v_1,\dots,v_m\}}$ be a subset in $\mathbb{R}^d$ of rank $d$ and let $r$ be a point in $\conv(V)$. Then the cover $\mathcal{F}$ is called homotopically non-trivial (with respect to $(V,r)$), if there exists $f:B^{k} \rightarrow \mathbb{R}^n$ such that $[(f^*\mathcal{F}|_{\partial B^{k}})_{(V,r)}] \neq 0$ or $[(f^*\mathcal{F}|_{\partial B^{k}})_{(V,r)}] =N/A$. 
\end{definition}

The following theorem is a consequence of Theorem \ref{VExtSph}:
\begin{theorem}
\label{NTC}
The closed (or open) cover $\mathcal{F}$ of $\mathbb{R}^n$ is homotopically non-trivial (with respect to $(V,r)$)  if there exists $S \in \bs(V,r)$ such that $\bigcap_{v_i \in S}F_i \neq \emptyset$. 
\end{theorem}
\begin{proof}
\begin{enumerate}
    \item Suppose there exists $S \in \bs(V,r)$  and a point $x \in \bigcap_{v_i \in S}F_i $. Define ${f: B^k \rightarrow x}$. Then there is balanced intersection on the boundary $\partial B^{k}$ and $[(f^*\mathcal{F}|_{\partial B^{k}})_{(V,r)}] = N/A$.
    \item In the case when there is $f: B^k \rightarrow \mathbb{R}^n$ such that $[(f^*\mathcal{F}|_{\partial B^{k}})_{(V,r)}] = N/A$ we already have a balanced intersection on the boundary of $B^k$. Assume $[(f^*\mathcal{F}|_{\partial B^{k}})_{(V,r)}] \neq 0$. Then from Theorem \ref{VExtSph} it follows that there is $S \in \bs(V,r)$ such that $\bigcap_{v_i \in S}f^*F_i \neq \emptyset $ and hence $\bigcap_{v_i \in S}F_i \neq \emptyset$.
\end{enumerate}

\end{proof}

\medskip 

For a closed cover of $\mathbb{R}^n$, we want to establish a connection between balanced intersections and zeros of a vector field. Using this analogy, we define an index of an isolated component of a balanced intersection in the case $n=d$.

To define the index rigorously, we restrict ourselves to the special case of closed covers in which, for any $S \in \bs(V,r)$, the intersection $\bigcap_{v_i \in S} F_i$ consists of isolated and compact connected components. We call such a cover $\mathcal{F}$ \textit{isolated}, and throughout this section we consider only isolated covers.

\medskip

Denote by $\bs(\mathcal{F},V,r)$ the set of all balanced intersections, i.e.
\[
\bs(\mathcal{F},V,r) \;=\; \bigcup_{S \in \bs(V,r)} \left(\bigcap_{v_i \in S} F_i\right).
\]
Since for any closed ball $B^n \subset \mathbb{R}^n$ there are only finitely many connected components of $\left(\bigcap_{v_i \in S} F_i\right) \cap B^n$, it follows that $\bs(\mathcal{F},V,r)$ also consists of isolated compact connected components.

\medskip
Let $C$ be a connected component of $\bs(\mathcal{F},V,r)$ for a closed cover $\mathcal{F}$, and let $C_{\varepsilon}$ denote its $\varepsilon$-neighborhood. There exists a number $r>0$ such that for every $\varepsilon<r$ we have 
\[
\partial C_{\varepsilon} \cap \bs(\mathcal{F},V,r)=\emptyset.
\] 
We will refer to such values of $\varepsilon$ simply as \emph{small enough $\varepsilon$}. The index of a component $C$ is then defined as the mapping degree of the cover induced on the boundary $\partial C_{\varepsilon}$. 

However, two technical issues arise:
\begin{enumerate}
    \item $\partial C_{\varepsilon}$ may fail to be a manifold, for instance when $C$ is a wild compact fractal.
    \item It is not immediately clear why the degree is independent of the choice of small enough $\varepsilon$.
\end{enumerate}

In what follows we address these technical problems.
The first issue can be resolved as follows:

\begin{claim}
    For every $\varepsilon > 0$ there exists a closed $n$-manifold $M \subset \mathbb{R}^n$ with smooth boundary $\partial M$ such that $C \subset M$ and $M \subset C_{\varepsilon}$. 
\end{claim}

\begin{proof}
    Consider the distance function $d: \mathbb{R}^n \to \mathbb{R}$ given by $d(x) = \mathrm{dist}(x, C)$. Let $\hat{d} : \mathbb{R}^n \to \mathbb{R}$ be a smooth approximation such that $|\hat{d}(x) - d(x)| < \tfrac{\varepsilon}{1000}$ for every $x \in \mathbb{R}^n$. Choose a regular value $\varepsilon/4 < \varepsilon_1 < \varepsilon/2$ and set 
    \[
    M = \hat{d}^{-1}((-\infty, \varepsilon_1]).
    \]
    Then $M \subset \mathbb{R}^n$ is a closed $n$-manifold with smooth boundary $\partial M = \hat{d}^{-1}(\varepsilon_1)$. For every $x \in C$ we have $\hat{d}(x) < \varepsilon/1000$, so $C \subset M$. For any $x \in M$ we have 
    \[
    d(x,C) < |\hat{d}(x)| + \tfrac{\varepsilon}{1000} < \varepsilon,
    \] 
    and therefore $M \subset C_{\varepsilon}$.
\end{proof}

The second issue is resolved by the following:

\begin{claim}
    Suppose there are two closed $n$-manifolds with smooth boundary $M_1$ and $M_2$ such that 
    \[
    C \subset M_1 \subset C_{\varepsilon_1} \subset M_2 \subset C_{\varepsilon_2}.
    \] 
    Then 
    \[
    \deg\!\left((\mathcal{F}|_{\partial M_1})_{(V,r)}\right) 
    = \deg\!\left((\mathcal{F}|_{\partial M_2})_{(V,r)}\right).
    \]
\end{claim}

\begin{proof}
    Consider 
    \[
    N = (M_2 \setminus M_1) \cup \partial M_1.
    \] 
    Then $N$ is a closed $n$-manifold with smooth boundary $\partial N = \partial M_1 \cup \partial M_2$. Since $C_{\varepsilon}$ contains no balanced intersections other than $C$, and $C \cap N = \emptyset$, it follows that there are no balanced intersections in $N$. Moreover, $\partial M_1$ inherits the opposite orientation in $N$. By Theorem~\ref{VExtMan} we obtain
    \[
    \deg\!\left((\mathcal{F}|_{\partial M_1^{-}})_{(V,r)}\right)
    + \deg\!\left((\mathcal{F}|_{\partial M_2})_{(V,r)}\right) = 0,
    \]
    and therefore
    \[
    \deg\!\left((\mathcal{F}|_{\partial M_1})_{(V,r)}\right)
    = \deg\!\left((\mathcal{F}|_{\partial M_2})_{(V,r)}\right).
    \]
\end{proof}

Denote by $\text{ind}(C_{(V,r)})$ the degree of the cover $\mathcal{F}|_{\partial M}$ for any closed $n$-manifold with smooth boundary $M \subset \mathbb{R}^n$ such that  \(C \subset M \subset C_{\varepsilon}\) for small enough $\varepsilon$.

Let $\mathrm{Conn}(\bs(\mathcal{F},V,r))$ denote the family of connected components of $\bs(\mathcal{F},V,r)$. With these definitions in place, we can state the following:

\begin{theorem}
    Let $\mathcal{F}$ be a closed isolated cover of $\mathbb{R}^n$, and let $V \subset \mathbb{R}^n$ be a set of points of rank $n$ with $r \in \conv(V)$. Let $M \subset \mathbb{R}^n$ be a closed $n$-manifold with smooth boundary. Suppose 
    \[
    \deg\!\left((\mathcal{F}|_{\partial M})_{(V,r)}\right) = k.
    \] 
    Then 
    \[
    \sum_{\substack{C \subset M \\ C \in \mathrm{Con}(\bs(\mathcal{F},V,r))}} \text{ind}(C_{(V,r)}) \;=\; k.
    \]
\end{theorem}

\begin{proof}
    The proof follows a standard argument. For each connected component $C$, remove the corresponding approximating manifold $M_{C}$ from $M$: 
    \[
    M' = M \setminus \bigcup M_C.
    \] 
    Then $M'$ is an $n$-dimensional manifold with smooth boundary such that $\mathcal{F}$ has no balanced intersections in $M'$. The cover $\mathcal{F}$ extends $\mathcal{F}|_{\partial M'}$ from the boundary to the whole $M'$, and by Theorem~\ref{VExtMan} it follows that 
    \[
    \deg\!\left((\mathcal{F}|_{\partial M'})_{(V,r)}\right) = 0.
    \] 
    On the other hand,
    \[
    \deg\!\left((\mathcal{F}|_{\partial M'})_{(V,r)}\right) 
    = \deg\!\left((\mathcal{F}|_{\partial M})_{(V,r)}\right) 
      - \sum \text{ind}(C_{(V,r)}).
    \] 
    Combining these two equalities gives the result.
\end{proof}

Note that the absolute value of the index $\text{ind}(C_{(V,r)})$ is the same for $\bs$-equivalent pairs $(V,r)$.


\section{KKMS and Sperner's lemmas as consequences of Theorem \ref{VExtSph}}

In the works \cite{MusH} and \cite{MusBC}, various corollaries were obtained from Theorem \ref{ExtSph}, for example (\cite{MusH}, Theorem 3.1), (\cite{MusH}, Cor. 3.2), and (\cite{MusBC}, Th. 4.2), which generalize classical discrete fixed point theorems. We will present the corollaries of these theorems here.

\medskip

\begin{theorem}
\label{A}
    Let $(V,r)$ be a pair in $\mathbb{R}^d$ of full rank such that $r$ is the mass center of $V$. Let $\mathcal{F}=\{F_1, \dots, F_m\}$ be a closed (or open) cover of an $n$-dimensional ball $B^n$ such that $[\mathcal{F}_{(V,r)}]$ is not homotopic to zero on the boundary. Then there exists a balanced set  $S\in\bs(V,r)$ such that the intersection $\bigcap_{v_i \in S} F_i$ is non-empty.
\end{theorem}

\medskip 

\begin{theorem}
\label{B}
  Let $(V,r)$ be a pair in $\mathbb{R}^d$ of full rank such that $r$ is the mass center of $V$. Let $T$ be a triangulation of an $n$-dimensional ball and let there be a coloring ${L:V(T)\to \{1,...,m\}}$ such that $[L_{(V, r)}]$ is not null-homotopic on the boundary. Then there exists a simplex $\sigma$ in $T$ and a set of indices $I\subset [m]$ such that $\{v_i\}_{i \in I}\in\bs(V)$ such that the vertices of $\sigma$ are colored with all colors from $I$. In other words, $\sigma $ is $r$-balanced simplex.
\end{theorem}

\medskip 

\begin{proof}
    These Theorems \ref{A} and \ref{B} are direct consequences of Theorem \ref{VExtSph}. 
\end{proof}

We note that the condition  ``homotopically non-trivial'' on the boundary holds for classical fixed-point theorems. In particular, these include Sperner's coloring,  KKM covering, and antipodal coloring on the boundary.

As a consequence of Theorems \ref{A} and \ref{B}, by studying balanced sets of vertices of polyhedra, one can obtain discrete versions of fixed-point theorems.

\medskip 

\begin{theorem}[KKM Lemma]
Let   $\Delta^{n-1} = \conv\{v_1, \dots, v_n\}$ be a simplex and let $\mathcal{C}=\{C_1, \dots, C_n\}$ be a closed cover of $\Delta^{n-1}$ such that for every face $\Delta_J = \conv\{v_j\}_{j \in J}$ we have $\Delta_J \subset \bigcup_{i \in J}C_j$. Then $\bigcap C_i \neq \emptyset$ .
\end{theorem}

\begin{proof}
Let $V$ be the set of vertices of $\Delta^{n-1}$ in $\mathbb{R}^{n}$ and let $r$ be the mass center of $V$. Note that $\bs(V,r)=\{V\}$. If $\mathcal{C}$ is a KKM cover, then $\deg(\mathcal{C}_{(V,r)})=1$. Thus, the lemma follows from Theorem \ref{A}.  
\end{proof}

\medskip 

Let $\Delta^{n-1}$ be a simplex of dimension $n$ with vertices $V=\{v_1, \dots, v_n\}$. $T$ be a triangulation of $\Delta^{n-1}$, and let $V(T)$ be the vertex set of the triangulation. We color each vertex $v_i$ with color $i$. We also have a coloring of the vertices of the triangulation, i.e., we have a mapping $L: V(T) \rightarrow \{1, \dots, n\}$, such that if a vertex $t$ lies inside a face $\Delta_{J}=\conv\{v_j\}_{j \in J}$, then $L(t) \in J$. Such a coloring is called Sperner's coloring. A Sperner's coloring can be similarly defined for any polyhedra.

\begin{theorem}[Sperner's Lemma]
For any Sperner coloring of a triangulation of the simplex $\Delta^{n-1}$  in $n$ colors there exists a simplex in the triangulation whose vertices are colored in all colors.
\end{theorem}

\begin{proof}
Let $V$ be the set of vertices of $\Delta^{n-1}$ in $\mathbb{R}^{n}$ and let $r$ be the mass center of $V$. Note that $\bs(V,r)=\{V\}$. Note that $\text{deg}(L_{(V,r)})=1$ and the theorem follows from Theorem \ref{B}

\end{proof}

\medskip 

The following theorem was obtained by Shapley in \cite{Sh}.

\begin{theorem}[KKMS Lemma]
    Consider the simplex $\Delta^{n-1}=\conv\{v_1, \dots, v_n\}$. Let $V=\{v_{I} \mid I \subset [n]\}$, where $v_I$ is the mass center of the face $\Delta_I=\conv({v_i})_{i \in I}$ and let $r$ be the mass center of the whole simplex $\Delta^{n-1}$.

    Let $C=\{C_{I}\}_{I \subset [n]}$ be a covering of the simplex $\Delta^{n-1}$ such that every face $\Delta_J$ is covered by the sets $\{C_I\}_{I \subset J}$. Then there exists a balanced set $S \subset V$  such that $r \in \conv(S)$ and  $\bigcap_{I \in S}C_I\neq \emptyset$.
\end{theorem}

\begin{proof}
    Note that $\deg(C_{(V,r)})=1$. Thus, the KKMS lemma is a direct consequence of Theorem \ref{A}.
\end{proof}

Alternative formulations and direct proofs of the above theorems can be found in different works, for example in \cite{DeLoeraPetersomSu,Herings,Komiya}.

\medskip

\textbf{Acknowledgement.} We wish to thank Roman
Karasev and Oleg Musin for helpful discussions and comments.

\end{document}